\documentclass[10pt]{amsart}

\usepackage{amsmath,amsthm,amssymb}
\usepackage{hyperref}
\usepackage{geometry}
\usepackage{tikz-cd}

\newtheorem{theorem}{Theorem}[section]
\newtheorem{lemma}[theorem]{Lemma}
\newtheorem{proposition}[theorem]{Proposition}
\newtheorem{corollary}[theorem]{Corollary}

\theoremstyle{definition}
\newtheorem{definition}[theorem]{Definition}

\begin{document}

\title{Uncountably many quasi-isometric torsion-free groups}
\author{Vladimir Vankov}
\address{School of Mathematics, University of Bristol, Bristol, UK and Heilbronn Institute for Mathematical Research, Bristol, UK}
\email{vlad.vankov@bristol.ac.uk}

\begin{abstract}
We construct uncountably many finitely generated, pairwise non-isomorphic torsion-free groups, all of which fall into the same quasi-isometry class.
This is done by considering Schur covering groups and group cohomology, with the necessary geometric ingredient coming from the theory of bounded-valued cohomology.
\end{abstract}

\maketitle

\tableofcontents

\section{Introduction}

A central theme within Geometric Group Theory is the study of coarse geometry.
\emph{Quasi-isometries} provide a natural framework for working with large-scale geometric features of finitely generated groups.
Much attention has been given to studying the possible nature of quasi-isometry classes. With regards to large sets, the first \emph{uncountable} family of different quasi-isometry classes of finitely generated groups was constructed by Grigorchuk \cite{grigorchuk}. This used periodic groups and hence contained torsion; the first such family consisting of torsion-free groups was built by Bowditch a decade later \cite{bowditch}.

More recently, uncountable families of quasi-isometrically distinguished finitely generated groups have been exhibited with an abundance of other properties, including: solvable groups \cite{cornuliertessera}, groups containing expanders (\cite{hume}, \cite{grubersisto}), groups with exotic finiteness properties \cite{krophollerlearysoroko}, groups admitting interesting acylindrical actions (\cite{abbotthume}, \cite{abbotthume2}), residually finite groups \cite{chongwise}, and groups with various combinations of several different kinds of dimension \cite{martinezpedrozasanchezsaldana}.

In addition, Minasyan, Osin and Witzel provided set-theoretic tools to enable the construction of uncountable sets of quasi-isometrically distinguished groups in general \cite{minasyanosinwitzel}.

The above results address the question of distinguishing large sets of \emph{different} quasi-isometry classes.
On the other hand, the (dual) question of studying \emph{one} large class has been relatively unexplored in the literature.
Erschler constructed uncountably many quasi-isometric but pairwise non-isomorphic finitely generated groups, showing that a single quasi-isometry class can itself be uncountable \cite{erschler}.
However, these groups are all virtually isomorphic (they contain finite normal subgroups such that the corresponding quotients have isomorphic finite-index subgroups).
This means that their quasi-isometries come from the fact that they are commensurable up to finite kernels (isomorphic to \(C_2\) in this case), and they contain lots of torsion due to being periodic.
A similar construction was also considered in \cite{weigelzaleskii}, where the finite kernels were taken to be \(C_p\) instead.

It is more difficult to construct quasi-isometries without relying on torsion or commensurability and finite-index subgroups.
Indeed, various classes of groups admit what is known as \textit{quasi-isometric rigidity}.
Any finitely generated groups quasi-isometric to \(\mathbb{Z}^n\) must contain \(\mathbb{Z}^n\) as a finite-index subgroup (\cite{pansu}, \cite{cornuliertesseravalette}).
Many right-angled Artin groups \(G\) have the property that any group \(H\) quasi-isometric to \(G\) must actually be commensurable to \(G\) \cite{huang}.
A stronger version of this is known for certain sub-families of right-angled Coxeter groups \cite{boundsxie} and solvable groups \cite{bourdonremy}\cite{dymarzfisherxie}, where two members of the family are quasi-isometric if and only if they are isomorphic.
Further similar instances, including different kinds of rigidity, are surveyed in \cite{hughesmartinezpedrozasanchezsaldana}.

Any quasi-isometry class of groups contains groups with torsion: for any group \(G\) within a quasi-isometry class, the group \(G\times H\) is also in the same quasi-isometry class for any finite group \(H\).
Nonetheless, one can ask how many finitely generated \emph{torsion-free} groups a quasi-isometry class contains (up to group isomorphism).
The aim of the present work is to exhibit the existence of a class for which there are many.

\begin{theorem}\label{thm:largeclass}
    There exist uncountably many finitely generated pairwise non-isomorphic torsion-free groups which are all quasi-isometric.
\end{theorem}

In fact, we provide a method to construct lots more such new examples.
The idea is to construct central extensions with properties controlled by group cohomology.
The cohomological conditions of \emph{distinguished}, \emph{rich} and \emph{weakly bounded} will be defined in the later sections.
Their purposes are to ensure sufficiently many extensions, the ability to tell apart isomorphism classes and to establish the quasi-isometries, respectively.

\begin{theorem}\label{thm:manyclasses}
    Suppose that a finitely generated group \(G\) has free abelian abelianisation.
    If \(H^2(G,\mathbb{Z})\) contains uncountably many distinguished classes which are both rich and weakly bounded, then among the corresponding central extensions of \(G\) by \(\mathbb{Z}\), there are uncountably many groups which are quasi-isometric but pairwise not isomorphic.
\end{theorem}

Furthermore, these groups will all be quasi-isometric to \(G\times\mathbb{Z}\) in particular.
A source of such groups \(G\) comes from small-cancellation groups which are not finitely presented, and the particular group used here for Theorem \ref{thm:largeclass} will come from this category.

In the spirit of the results in the literature featuring uncountably many different quasi-isometry classes of groups with interesting properties, we would like to initiate the study of large quasi-isometry classes of groups with some fixed property.
For which properties \(\mathcal{P}\) do there exist uncountably many quasi-isometric groups all having \(\mathcal{P}\)?
One property of particular interest is that of being \emph{almost finitely presented}, also known as being of type \(FP_2\), characterised by having a finitely generated relation module.
Unlike finitely presented groups, of which there are only countably many, there are uncountably many groups of type \(FP_2\).

\medskip

\textbf{Question}: Does there exist an uncountable set of pairwise non-isomorphic groups of type \(FP_2\) which are all quasi-isometric?

\medskip

The strategy behind Theorem \ref{thm:manyclasses} relies on uncountably many cohomology classes in \(H^2(G,\mathbb{Z})\).
Since groups of type \(FP_2\) must necessarily have finitely generated second cohomology (\cite{brown}, Chapter VIII), this does not apply to the question.

Note that any property \(\mathcal{P'}\) for which there can be at most countably many groups satisfying it, will necessarily not be quasi-isometry invariant among groups with \(\mathcal{P}\), if an appropriate uncountable quasi-isometry class of groups with \(\mathcal{P}\) were to exist.
This does not necessarily have to be true if uncountably many groups with \(\mathcal{P}\) exist.
For example, it is not known whether the word problem is a quasi-isometry invariant among groups of type \(FP_2\).

\medskip

\textbf{Outline.}

Section 2 contains the necessary background on quasi-isometries, central extensions and the relationship between them through \(\ell^{\infty}\)-cohomology. Section 3 studies Schur covering groups and contains the proof of Theorem \ref{thm:manyclasses}. Section 4 exhibits a particular finitely generated group with the correct properties to deduce the statement of Theorem \ref{thm:largeclass}.

\medskip

\textbf{Acknowledgements.}

The study of large quasi-isometry classes was first brought to the author's attention after a conversation between Peter Kropholler and Ian Leary, following the release of \cite{krophollerlearysoroko}.
This work was supported by the Additional Funding Programme for Mathematical Sciences, delivered by EPSRC (EP/V521917/1) and the Heilbronn Institute for Mathematical Research.

\section{Central extensions and quasi-isometries}

\begin{definition}[Quasi-isometry]
    Let $(X,d_1)$ and $(Y,d_2)$ be metric spaces.
    A function $f:X\to Y$ is called a \emph{quasi-isometric embedding} if
    \[
    \exists C\geqslant 1, K\geqslant 0 : \forall a,b\in X, \ \frac{1}{C}d_1(a,b)-K\leqslant d_2\left(f(a),f(b)\right)\leqslant Cd_1(a,b)+K.
    \]
    If in addition $f$ is also a \emph{coarse surjection}, i.e. if
    \[
    \exists \epsilon\geqslant 0 : \forall y\in Y\ \exists x\in X : d_2(f(x),y)\leqslant \epsilon,
    \]
    then $f$ is called a \emph{quasi-isometry}.
    If such an $f$ exists, the spaces $X$ and $Y$ are called \emph{quasi-isometric.}
    This gives an equivalence relation on metric spaces.
\end{definition}

If $G$ is a finitely generated group, then given some finite generating set $S$ we can build the Cayley graph of $G$ and consider the corresponding word metric.
The choice of $S$ does not matter for large-scale geometry, i.e. any such two Cayley graphs will be quasi-isometric.
We say that two finitely generated groups $G$ and $H$ are quasi-isometric if their Cayley graphs are.
This gives an equivalence relation on finitely generated groups.

\begin{definition}[Extension]
    An \emph{extension group} of a group \(G\) is a group \(E\) equipped with a surjective homomorphism \(\phi:E\to G\).
    If we denote the kernel of this map by \(K\), then we have a short exact sequence
    \begin{center}
    \begin{tikzcd}
    1 \arrow[r] & K \arrow[r] & E \arrow[r,"\phi"] & G \arrow[r] & 1.
    \end{tikzcd}
    \end{center}
    We refer to this sequence as the \emph{extension} of \(G\) by \(K\) (sometimes called a \emph{group extension}).
\end{definition}

\begin{definition}[Central and stem extensions]
    Denote the centre of a group \(G\) by \(Z(G)\).
    If in the above, \(K\subseteq Z(E)\), then we call the corresponding extension a \emph{central extension} of \(G\) by \(K\).

    Denote by \([G,G]\) the derived subgroup of \(G\) (sometimes called the \emph{commutator subgroup}).
    If in the above, we have \(K\subseteq Z(E)\cap[E,E]\), then we call the corresponding extension a \emph{stem} extension.
\end{definition}

Recall that there is a correspondence between central extensions of \(G\) by abelian groups \(A\) and \(H^2(G,A)\), with \(A\) viewed as a trivial \(G\)-module (trivial action of \(G\) on \(A\)) (\cite{brown}, Chapter IV). More concretely, the cohomology classes in \(H^2(G,A)\) classify extension groups up to the equivalence \(E\sim E'\) given by isomorphisms \(\phi\) making the diagram

\begin{equation}\label{eq:equiv}
\begin{tikzcd}
    & & E \arrow[dr]\arrow[dd,"\phi"]& &\\
    1\arrow[r] & A\arrow[ur]\arrow[dr] & & G\arrow[r] & 1\\
    & & E'\arrow[ur] & & \\
\end{tikzcd}
\end{equation}

commute.
Note that given two different elements of \(H^2(G,A)\), the two different classes of central extension do not necessarily have non-isomorphic extension groups.
This is primarily why more work is required in the next section to properly distinguish uncountably many groups for our purposes.

Next, we turn our attention to establishing quasi-isometries between torsion-free groups.

It was noticed by Gersten that an extension group \(E\) of \(G\) by \(A\) coming from a central extension is quasi-isometric to \(G\times A\) when the corresponding cohomology class is \emph{bounded} \cite{gersten}.
This means that the class lies in the image of the comparison map from bounded cohomology
\[
H^2_b(G,A)\to H^2(G,A).
\]

However, Neumann and Reeves pointed out that the argument did not necessarily require the cohomology class to be bounded \cite{neumannreeves}.
Instead, all that was required was for the cohomology class to be \emph{weakly bounded}.
This means that the class lies in the kernel of the natural map to \(\ell^{\infty}\)-cohomology
\[
H^2(G,A)\to H^2_{\infty}(G,A).
\]

This was clarified by Frigerio and Sisto in \cite{frigeriosisto}, where Proposition 1.11 made the characterisation of weakly bounded classes in terms of \(\ell^{\infty}\)-cohomology concrete.
They also exhibited the first example of a finitely generated group with weakly bounded classes which were not bounded.
A finitely presented example was later given in \cite{ascarimilizia}.

We shall not develop the theories of \(\ell^{\infty}\)-cohomology (sometimes also known as \emph{bounded-valued} cohomology) and bounded cohomology further here, as the main tool that we will require is the following proposition.

\begin{proposition}[Quasi-isometries via bounded-valued cohomology]\label{prop:qi}
    A central extension group \(E\) of \(G\) by \(\mathbb{Z}\) is quasi-isometric to \(G\times \mathbb{Z}\) if and only if the corresponding cohomology class in \(H^2(G,\mathbb{Z})\) is weakly bounded. 
\end{proposition}

\begin{proof}
    Implicit in \cite{gersten}, \cite{neumannreeves}. Also appears as Theorem 1.8 in \cite{kleinerleeb}.
\end{proof}

Thus groups with many different weakly bounded classes will have many quasi-isometric central extensions.
It will be the aim of the following section to distinguish such central extensions up to group isomorphism.
The discussion of showing which classes are weakly bounded is postponed until the final section.

We end this section with a widely well-known result.
However, it is a key step in our argument and its proof is short, thus we include it here for completeness.

\begin{lemma}\label{lem:torsionfree}
    Extensions of torsion-free groups by torsion-free groups are again torsion-free.
\end{lemma}

\begin{proof}
    Suppose that \(E\) is an extension of the group \(G\) with kernel \(K\).
    We have an exact sequence
    \begin{center}
    \begin{tikzcd}
    1 \arrow[r] & K \arrow[r,"\phi"] & E \arrow[r,"\psi"] & G \arrow[r] & 1.
    \end{tikzcd}
    \end{center}
    If \(g\in E\) has finite order, then because \(\psi\) is a homomorphism, so does \(\psi(g)\in G\).
    As \(G\) is torsion-free, this means that the order of \(\psi(g)\) must be 1, i.e. \(g\) lies in the kernel of \(\psi\).
    By exactness, there must be some \(h\in K\) such that \(\phi(h)=g\).
    Now \(\phi\) is injective, which forces \(h\) to have finite order too.
    As \(K\) is torsion-free, this means the order of \(h\) must be 1, and \(\phi\) being a homomorphism now forces \(g\) to be the identity as well.
\end{proof}

In particular, this applies to central extensions of torsion-free groups by \(\mathbb{Z}\).

\section{Quotients of Schur covering groups}

The proof of Theorem \ref{thm:manyclasses} will utilise quotients of a much larger central extension.

\begin{definition}
    The \emph{Schur multiplier} of a group \(G\) is the second homology group \(H_2(G,\mathbb{Z})\).
\end{definition}

We will denote by \(G^{\text{ab}}\) the abelianisation \(G/[G,G]\), which is the same as \(H_1(G,\mathbb{Z})\).

Recall that the Universal Coefficients Theorem provides a short exact sequence

\begin{equation}\label{eq:uct}
\begin{tikzcd}
    0 \arrow[r] & \operatorname{Ext}^1_{\mathbb{Z}}(G^{\text{ab}},A)\arrow[r] & H^2(G,A) \arrow[r] & \operatorname{Hom}(H_2(G,\mathbb{Z}),A)\arrow[r] & 0
\end{tikzcd}
\end{equation}

for any abelian group \(A\) (with trivial \(G\)-action).

For \(G\) perfect (having trivial abelianisation) or \(G^{\text{ab}}\) being isomorphic to some \(\mathbb{Z}^n\), \(\operatorname{Ext}^1_{\mathbb{Z}}(G^{\text{ab}},A)\) will vanish.
The above sequence then yields an isomorphism
\begin{equation}\label{eq:cong}
H^2(G,A) \cong \operatorname{Hom}(H_2(G,\mathbb{Z}),A).
\end{equation}

We will focus on the case when \(A=H_2(G,\mathbb{Z})\), the Schur multiplier of \(G\).

\begin{definition}[Unique Schur cover]
    We say that a group \(G\) has a \emph{unique Schur covering group} if
    \[
    \operatorname{Ext}^1_{\mathbb{Z}}(H_1(G,\mathbb{Z}),H_2(G,\mathbb{Z}))
    \]
    is trivial.
\end{definition}

For a group \(G\) with a unique Schur covering group, using the isomorphism (\ref{eq:cong}) we can associate central extensions of \(G\) by \(H_2(G,\mathbb{Z})\) (via second cohomology classes with coefficients being the Schur multiplier) with elements of \(\operatorname{Hom}(H_2(G,\mathbb{Z}),H_2(G,\mathbb{Z}))\).
We define the \emph{Schur covering group} to be the extension group \(E\) in the extension

\begin{equation}\label{eq:schur}
\begin{tikzcd}
    1 \arrow[r] & H_2(G,\mathbb{Z})\arrow[r] & E\arrow[r] & G\arrow[r] & 1
\end{tikzcd}
\end{equation}

that corresponds to the identity map \(H_2(G,\mathbb{Z})\to H_2(G,\mathbb{Z})\) under the above construction.

For the case when \(G\) is perfect, this is known as the \emph{universal central extension}.

The following lemma is one of the reasons why we want Schur coverings.

\begin{lemma}[Schur coverings are stem]\label{lem:stem}
    The particular extension (\ref{eq:schur}) is a stem extension.
\end{lemma}

\begin{proof}
    The Hochschild-Serre spectral sequence associated to (\ref{eq:schur}) gives rise to a 5-term exact sequence (also known as the Stallings exact sequence \cite{stallings})
    \begin{center}
    \begin{tikzcd}
    H_2(E,\mathbb{Z})\arrow[r] & H_2(G,\mathbb{Z})\arrow[r,"\phi"] & H_2(G,\mathbb{Z})\arrow[r,"c"] & H_1(E,Z)\arrow[r,"\psi"] & H_1(G,\mathbb{Z})
    \end{tikzcd}
    \end{center}
    with the map \(\phi\) being precisely the element of \(\operatorname{Hom}(H_2(G,\mathbb{Z}),A)\), with \(A\) being the Schur multiplier of \(G\), that corresponds to the extension (\ref{eq:schur}) under the isomorphism (\ref{eq:cong}).
    By construction, this was chosen to be the identity map, therefore \(\phi\) is the identity homomorphism on \(H_2(G,\mathbb{Z})\).
    By exactness, the map \(c\) must hence be the zero map.
    This implies that the map \(\psi\) is injective.
    But the kernel of \(\psi\) is precisely
    \[
    H_2(G,\mathbb{Z})/(H_2(G,\mathbb{Z})\cap[E,E]).
    \]
    Since this must be the trivial group, it follows that \(H_2(G,\mathbb{Z})\subseteq[E,E]\), i.e. the kernel of the extension in (\ref{eq:schur}) is in both the centre and the derived subgroup of the extension group.
\end{proof}

To properly distinguish isomorphism classes of groups, we will distinguish many quotients. For this, we will require finite generation of our big extension.
This is why having a stem extension is helpful.

\begin{proposition}[Central commutators are non-generators]
For a group \(G\), let 
\[
A:=[G,G]\cap Z(G).
\]
If \(S\cup B\) generates \(G\) for some subsets \(S\subseteq G\setminus A\) and \(B\subseteq A\), then \(S\) generates \(G\).
\end{proposition}

\begin{proof}
We will show that every element of \(B\) may be written as a finite product of elements from \(S\).

Suppose that \(g\in B\).
Since \(B\) is contained in the derived subgroup of \(G\), we may write \(g=[a_1,b_1][a_2,b_2]\cdots[a_k,b_k]\) for some elements \(a_i,b_i\in G\). As \(S\cup B\) generates \(G\), we may write
\[
a_i=r_{i,1} r_{i,2} \cdots r_{i,m(i)} g_{i,1} g_{i,2} \cdots g_{i,p(i)},
\]
\[
b_i=s_{i,1} s_{i,2} \cdots s_{i,n(i)} h_{i,1} h_{i,2} \cdots h_{i,q(i)},
\]
for some elements \(r_{i,j},s_{i,j}\in S\) and \(g_{i,j},h_{i,j}\in B\) , where we used the fact that \(B\) is contained in the centre of \(G\) to move the \(g_{i,j}\) and \(h_{i,j}\) to the right of the expressions. Let \(r_i=\prod_{j=1}^{m(i)} r_{i,j}\) and \(s_i=\prod_{j=1}^{n(i)} s_{i,j}\). Then again using the fact that \(B\) is central, we get
\[
g=\prod_{i=1}^k[a_i,b_i]=\prod_{i=1}^k[r_i,s_i].
\]
Thus every element of \(G\) may be written as a finite product of elements from \(S\).
\end{proof}

\begin{corollary}\label{cor:fingen}
    If a group \(G\) is generated by some set \(S\), then any stem extension \(E\) of \(G\) can also be generated by a set of the same cardinality as \(S\).
\end{corollary}

This allows us to put large non-finitely generated abelian groups into the center, while preserving finite generation.

Of course, any stem extension would be enough to ensure finite generation. Next, we look at how having a Schur covering in particular helps us to control the kind of extensions that we get as quotients.

\begin{definition}[Kronecker pairing]
    Evaluation of 2-cochains on 2-chains yields a map
    \[
    \left<\cdot,\cdot\right>\colon H^2(G,\mathbb{Z})\times H_2(G,\mathbb{Z})\to \mathbb{Z}.
    \]
    This is the same regardless of the resolution used, since the Universal Coefficients Theorem with coefficients as \(\mathbb{Z}\) in (\ref{eq:uct}) provides a map
    \begin{equation}\label{eq:kronecker}
    \phi\colon H^2(G,\mathbb{Z})\to \operatorname{Hom}(H_2(G,\mathbb{Z}),\mathbb{Z}).
    \end{equation}
    Thus any cohomology class \([a]\in H^2(G,\mathbb{Z})\) gives a homomorphism
    \[
    \phi([a])\colon H_2(G,\mathbb{Z})\to \mathbb{Z},
    \]
    which determines
    \[
    \left<[a],[b]\right>=\phi([a])([b])
    \]
    for \([b]\in H_2(G,\mathbb{Z})\).
\end{definition}

\begin{proposition}[Pushout]\label{prop:pushout}
    Suppose that \(G\) has a unique Schur covering group and \([a]\in H^2(G,\mathbb{Z})\).
    Then there is a commutative diagram
    \begin{equation}\label{eq:pushout}
    \begin{tikzcd}
        1 \arrow[r] & H_2(G,\mathbb{Z}) \arrow[r]\arrow[d,"f"] & E \arrow[r]\arrow[d,"\phi_f"] & G \arrow[r]\arrow[d,equal] & 1\\
        1 \arrow[r] & \mathbb{Z} \arrow[r] & E_f \arrow[r] & G \arrow[r] & 1,
    \end{tikzcd}
    \end{equation}
    where the top row is the unique Schur cover of \(G\) and the bottom row is a central extension of \(G\) by \(\mathbb{Z}\) corresponding to \([a]\in H^2(G,\mathbb{Z})\).
\end{proposition}

\begin{proof}
    The map \(f\) is given by the Kronecker pairing (i.e. this is \(\phi([a])\) under the mapping (\ref{eq:kronecker})).
    The rest of the diagram is the standard pushout construction of \(H_2(G,\mathbb{Z})\) along the maps \(f\) and \(H_2(G,\mathbb{Z})\to E\) from the stem extension of the Schur covering.
    The group \(E_f\) can be taken to be the largest quotient of \(\mathbb{Z}\rtimes E\) such that the left square of the diagram commutes (\cite{brown}, Chapter IV, Exercise 3.1(b)).
    Note that the bottom row is unique up to equivalence of central extensions, as in (\ref{eq:equiv}).
\end{proof}

This was the main reason we constructed the Schur cover: to relate the Schur covering group \(E\) to the central extension groups \(E_f\) via (\ref{eq:pushout}).
Observe that in general, the corresponding map \(\phi_f\) is not necessarily a surjective homomorphism.
We need more conditions to realise \(E_f\) as a quotient of \(E\).

\begin{definition}[Rich class]
    We say that a cohomology class \([a]\in H^2(G,\mathbb{Z})\) is \emph{rich} if the corresponding Kronecker pairing map
    \[
    f\colon H_2(G,\mathbb{Z})\to\mathbb{Z}
    \]
    coming from (\ref{eq:kronecker}) is surjective.
\end{definition}

Note that even between finitely generated groups, different homomorphisms can have the same kernels.

\begin{definition}[Distinguished classes]
    We say that two cohomology classes \([a],[b]\in H^2(G,\mathbb{Z})\) are \emph{distinguished} if their corresponding two Kronecker pairing maps
    \[
    f_{[a]}\colon H_2(G,\mathbb{Z})\to\mathbb{Z}
    \]
    and
    \[
    f_{[b]}\colon H_2(G,\mathbb{Z})\to\mathbb{Z}
    \]
    have different kernels.
    We say that a set of cohomology classes is distinguished, if the classes in the set are all pairwise distinguished.
\end{definition}

The final ingredient is finite generation to distinguish quotients coming from many normal subgroups.

\begin{lemma}[\cite{delaharpe}, Lemma 42]\label{lem:quotients}
    Suppose that \(\mathcal{S}\) is a set of normal subgroups of a finitely generated group \(G\), and \(\mathcal{Q}\) is the set of corresponding quotient groups.
    Then \(\mathcal{S}\) has uncountable cardinality if and only if there are uncountably many isomorphism classes of groups represented in \(\mathcal{Q}\).
\end{lemma}

\begin{proof}
    If there are uncountably many groups in \(\mathcal{Q}\), then they must have come from uncountably many quotient maps with uncountably many different set-wise kernels, thus \(\mathcal{S}\) must be uncountable.

    For the converse, suppose that \(\mathcal{S}\) is uncountable and consider each quotient group \(Q\in\mathcal{Q}\).
    Since \(G\) is finitely generated, each \(Q\) must also be finitely generated.
    Furthermore, \(G\) and each \(Q\) are countable.
    Thus there are at most countably many homomorphisms from \(G\) onto the isomorphism class of \(Q\), corresponding to at most countably many set-wise kernels.
    This means that at most countably many of the normal subgroups in \(\mathcal{S}\) give quotients isomorphic to \(Q\).
    It follows that \(\mathcal{Q}\) must contain uncountably many pairwise non-isomorphic groups.
\end{proof}

We are now ready to prove Theorem \ref{thm:manyclasses}.

\begin{proof}[Proof of Theorem \ref{thm:manyclasses}]
    Let \(\mathcal{S}\) be an uncountable set of distinguished classes from \(H^2(G,\mathbb{Z})\) which are all rich and weakly bounded.
    Because \(G^{\text{ab}}\) is free abelian, \(G\) will have a unique Schur covering group \(E\).
    Given a particular class \([a]\) from \(\mathcal{S}\), we apply Proposition \ref{prop:pushout} and consider the corresponding commutative diagram (\ref{eq:pushout}).
    Since \([a]\) is rich, the corresponding map \(f\) will be surjective.
    Then by the 4-lemma, the map \(\phi_f\) is surjective and \(E_f\) is the quotient of \(E\) by the kernel of \(\phi_f\).
    The kernel of \(\phi_f\) is precisely the image under the map \(H_2(G,\mathbb{Z})\to E\) of the kernel of \(f\).
    Since the classes in \(\mathcal{S}\) are distinguished, these kernels will all be pairwise distinct.
    By Lemma \ref{lem:stem} and Corollary \ref{cor:fingen}, \(E\) is finitely generated because \(G\) is finitely generated.
    Therefore by Lemma \ref{lem:quotients}, among the corresponding central extension groups \(E_f\) of \(G\) by \(\mathbb{Z}\), there will be uncountably many pairwise non-isomorphic groups.
    Finally, as all of the classes are weakly bounded, by Proposition \ref{prop:qi} these will all be quasi-isometric to each other and to \(G\times\mathbb{Z}\).
\end{proof}

We believe that some of the conditions in Theorem \ref{thm:manyclasses} could be relaxed.
For example, if \(G\) does not have a unique Schur covering group, then the map
\[
H^2(G,\mathbb{Z})\to\operatorname{Hom}(H_2(G,\mathbb{Z}),\mathbb{Z})
\]
may have a kernel, and different 2-cohomology classes may map to the same central extension under the above construction.
Nonetheless, we believe that in certain special cases, with enough care, one may still be able to build uncountably many isomorphism classes of extension groups \(E_f\) in (\ref{eq:pushout}).
Extensions of \(G\) by groups other than \(\mathbb{Z}\) should work too.
However, we do not pursue this here.

\section{Uncountably many rich cohomology classes}

In order to prove Theorem \ref{thm:largeclass}, it now suffices to provide a torsion-free group \(G\) that satisfies the conditions of Theorem \ref{thm:manyclasses}.
By Lemma \ref{lem:torsionfree}, the resulting uncountably many quasi-isometric groups will be all torsion-free.

\begin{proof}[Proof of Theorem \ref{thm:largeclass}]
For this purpose, we can actually use the group \(G\) from \cite{frigeriosisto} which was shown to have weakly bounded but not bounded classes.
This is given by the infinite presentation
\[
\left<a_1,a_2,a_3,a_4,t_1,t_2,t_3,t_4\ \middle|\ r_i : i\in\mathbb{N}_0\right>,
\]
where the relators are given by
\[
r_i=[t^i_1a_1t_1^{-i},t^i_2a_2t_2^{-i}][t^i_3a_3t_3^{-i},t^i_4a_4t_4^{-i}].
\]
By definition, this is finitely generated.
By (\cite{frigeriosisto}, Lemma 4.2), this group is \(C'(1/7)\) small-cancellation and no relation is a proper power.
This means it is torsion-free, and the presentation complex \(X\) is a classifying space for \(G\).
This can then be used to compute the homology and cohomology of \(G\).

By (\cite{frigeriosisto}, Proposition 4.3), there is an isomorphism
\begin{equation}\label{eq:sequences}
\psi\colon H^2(G,\mathbb{Z})\cong\mathbb{Z}^{\mathbb{N}}.
\end{equation}
This comes from the following.

Using similar notation to \cite{frigeriosisto}, we have 2-cells \(c_i\) in \(X\) corresponding to the relations \(r_i\).
These define classes \([c_i]\in H_2(G,\mathbb{Z})\).
For every \(\alpha\in H^2(G,\mathbb{Z})\) one can now apply the Kronecker pairing to get
\[
\alpha_i := \left<\alpha,[c_i]\right>.
\]
The isomorphism \(\psi\) in (\ref{eq:sequences}) is then given by
\[
\psi(\alpha)=(\alpha_0,\alpha_1,\dots).
\]
Thus we can think of classes in \(H^2(G,\mathbb{Z})\) as integer sequences.
By (\cite{frigeriosisto}, Theorem 4.10), any bounded sequence will correspond to a weakly bounded class (note that they have a milder condition needed for the corresponding class to be weakly bounded, but the sequence being bounded is sufficient).

Consider the set \(\mathcal{S}\) of all possible integer sequences with values in \(\{0,1\}\), which are not the all-zeroes sequence.

The cardinality of \(\mathcal{S}\) is uncoubtable, thus \(\psi^{-1}(\mathcal{S})\) consists of uncountably many classes in \(H^2(G,\mathbb{Z})\).

Each sequence in \(\mathcal{S}\) is bounded, hence \(\psi^{-1}(\mathcal{S})\) consists of weakly bounded classes.

Since we have removed the all-zeroes sequence, every sequence in \(\mathcal{S}\) contains a 1.
This means that by construction, the image of the Kronecker pairing map of every element in \(\psi^{-1}(\mathcal{S})\) contains the value 1.
Therefore it surjects onto \(\mathbb{Z}\), and every class in \(\psi^{-1}(\mathcal{S})\) is rich.

Finally, each pair of sequences \(\mathcal{A},\mathcal{B}\) in \(\mathcal{S}\) must differ in at least one place.
This means the Kronecker pairing evaluations of \(\psi^{-1}(\mathcal{A})\) and \(\psi^{-1}(\mathcal{B})\) differ on some class \([c_i]\in H_2(G,\mathbb{Z})\).
Since these two distinct values will be necessarily 0 and 1, this means that they are distinguished: one will contain \([c_i]\) in its kernel, and the other will not.
This applies to any pair of sequences, thus \(\psi^{-1}(\mathcal{S})\) consists of distinguished classes.

\end{proof}

\bibliographystyle{alpha}
\bibliography{references}

\end{document}